\pdfoutput=1
\documentclass[11pt]{article}

\usepackage[T1]{fontenc}
\usepackage[utf8]{inputenc}
\usepackage{lmodern}

\usepackage[margin=1in]{geometry}

\usepackage{amsmath,amssymb,amsthm}

\usepackage{microtype}
\usepackage[colorlinks=true,
            citecolor=blue,
            linkcolor=blue,
            urlcolor=blue]{hyperref}

\usepackage{enumitem}
\usepackage{booktabs}
\usepackage{listings}
\usepackage{xcolor}

\theoremstyle{plain}
\newtheorem{theorem}{Theorem}[section]
\newtheorem{lemma}[theorem]{Lemma}
\newtheorem{proposition}[theorem]{Proposition}
\newtheorem{corollary}[theorem]{Corollary}
\theoremstyle{definition}
\newtheorem{definition}[theorem]{Definition}

\theoremstyle{remark}
\newtheorem{remark}[theorem]{Remark}

\lstdefinestyle{lean}{%
  language={},
  basicstyle=\ttfamily\small,
  breaklines=true,
  breakatwhitespace=false,
  columns=flexible,
  keepspaces=true,
  showstringspaces=false,
  xleftmargin=2em,
  aboveskip=0.8\medskipamount,
  belowskip=0.8\medskipamount,
  mathescape=true,
}

\newcommand{\ZZ}{\mathbb{Z}}

\newcommand{\li}[1]{\texorpdfstring{\texttt{#1}}{#1}}
\newcommand{\lf}[1]{\texorpdfstring{\nolinkurl{#1}}{#1}}
\newcommand{\doi}[1]{\url{https://doi.org/#1}}

\title{A Prime-Generated Formalization of Nagata's\\
       Factoriality Theorem in Lean~4}
\author{
Arthur F.~Ramos\thanks{Corresponding author.}\\
Microsoft, USA\\
\texttt{arfreita@microsoft.com}
\and
Ruy J.~G.~B.~de~Queiroz\\
Centro de Inform\'atica, Universidade Federal de Pernambuco, Brazil\\
\texttt{ruy@cin.ufpe.br}
\and
Anjolina G.~de~Oliveira\\
Centro de Inform\'atica, Universidade Federal de Pernambuco, Brazil\\
\texttt{ago@cin.ufpe.br}
}
\date{}

\begin{document}
\maketitle


\begin{abstract}
We present, to our knowledge, the first public Lean~4\,/\,Mathlib
formalization of Nagata's factoriality theorem: if~$R$ is a noetherian domain
and $S\subseteq R$ is a prime-generated submonoid such that $S^{-1}R$ is a
UFD, then~$R$ is a UFD.  The development is organized around reusable transfer
lemmas and packages the theorem both for the concrete type
\li{Localization\,S} and at the abstract \li{IsLocalization} level.  As
applications, we formalize two structurally distinct Nagata-based proofs that
$R[X]$ is a UFD whenever~$R$ is a noetherian UFD: one via Laurent-polynomial
localization at powers of~$X$, and one via localization at the constant primes
and comparison with~$\operatorname{Frac}(R)[X]$.  The formalization also
replaces a degenerate ``prime-or-unit'' variant by the mathematically correct
prime-generated hypothesis exposed by the proof effort, and it yields the
iterated polynomial corollary $R[X][Y]$.
\end{abstract}

\paragraph{Keywords.}
Formalized mathematics, Lean~4, Mathlib, commutative algebra, localization,
unique factorization domains, Nagata's theorem.

\paragraph{MSC2020.}
Primary 03B35; Secondary 13A05, 13B30.

\section{Introduction}\label{sec:intro}

Nagata's factoriality theorem is a classical tool in commutative algebra that
recovers unique factorization in a base ring from unique factorization in a
suitable localization.  In a standard textbook
presentation~\cite{nagata1962,samuel1964,matsumura1989,kaplansky1970}, the
theorem states that if~$R$ is a noetherian integral domain, $S\subseteq R$ is
a multiplicative set whose elements factor into primes in~$R$, and the
localization~$S^{-1}R$ is a UFD, then $R$~itself is a UFD.  The argument
connects localization, irreducibility, primality, and noetherianity in a way
that is both algebraically instructive and technically nontrivial.

The theorem occupies a distinctive position in the landscape of formalization
targets.  It is short enough to be self-contained---the informal proof fits
comfortably on a single page---yet it requires careful coordination of several
pieces of algebraic infrastructure: localization of rings at a submonoid,
transfer of divisibility across the localization map, interaction between
irreducibility and primality in integral domains, and the characterization of
UFDs as noetherian domains in which every irreducible is prime.  It is also a
gateway result: one of its immediate applications is to prove that $R[X]$ is a
UFD whenever $R$~is, via localization at powers of the indeterminate.

\paragraph{Motivation and scope.}
We describe a Lean~4 formalization, built on Mathlib~\cite{mathlib}, that
packages a complete proof of Nagata's theorem together with a worked
application establishing unique factorization for polynomial rings.  The
formalization grew out of an attempt to close a gap in Mathlib's coverage of
classical factoriality results: while the library provides extensive
infrastructure for unique factorization, localization, and polynomial algebra,
no Nagata-style transfer theorem was available.

Nagata's theorem is an attractive formalization target for several reasons.
It is compact enough to be self-contained in a modest development, yet it
forces nontrivial interaction between three pillars of commutative algebra:
localization, the noetherian condition, and unique factorization.  The proof
demands transfer lemmas that move divisibility, irreducibility, and primality
across the localization map---exactly the kind of multi-step reasoning where
formalization provides high assurance.  Moreover, its immediate corollary for
polynomial rings makes it a practically useful building block.

During the formalization we found that an earlier prime-or-unit version of the
statement collapses for submonoids with more than one non-unit prime
generator.  The present development therefore adopts the mathematically correct
\emph{prime-generated} hypothesis and rebuilds the supporting machinery around
that formulation.  Section~\ref{sec:background} explains the algebraic issue,
and Section~\ref{sec:lessons} returns to the proof-engineering lesson.

\paragraph{Contributions.}
The main contributions are:
\begin{enumerate}[nosep]
\item To the best of our knowledge, the first public formalization of Nagata's
  factoriality theorem, stated under the correct prime-generated hypothesis on
  the multiplicative set~(\S\ref{sec:statement}).
\item A reusable theorem surface that combines concrete
  \li{Localization\,S} statements, abstract \li{IsLocalization} formulations,
  and packaged entry points for prime-generator closures
  ~(\S\ref{sec:statement}).
\item Two structurally distinct Nagata-based routes to the polynomial UFD
  theorem---via Laurent localization and via the fraction field---together with
  an iterated polynomial corollary obtained by reusing the same package
  ~(\S\ref{sec:application}).
\end{enumerate}
The development also shows how formalization can sharpen hypothesis design: the
shift from prime-or-unit to prime-generated was not cosmetic, but required to
state the theorem at the right level of generality.

\paragraph{Outline.}
Section~\ref{sec:background} reviews the mathematical background.
Section~\ref{sec:statement} gives the formal statement.
Section~\ref{sec:architecture} describes the proof architecture.
Section~\ref{sec:lean-statements} presents the key Lean declarations verbatim.
Section~\ref{sec:organization} discusses file organization and API design.
Section~\ref{sec:application} presents the two main applications.
Section~\ref{sec:lessons} draws proof-engineering and mathematical lessons.
Section~\ref{sec:related} surveys related work.
Section~\ref{sec:upstream} discusses prospective Mathlib contributions.
Section~\ref{sec:artifact} describes the artifact.
Section~\ref{sec:conclusion} concludes.

\section{Mathematical background}\label{sec:background}

We briefly recall the algebraic ingredients that enter the Nagata argument,
with an emphasis on the points most relevant to formalization.

\subsection{Localization at a multiplicative submonoid}

Let $R$ be a commutative ring and $S\subseteq R$ a multiplicative submonoid
(i.e., $1\in S$ and $S$ is closed under multiplication).  The
\emph{localization} of~$R$ at~$S$, written~$S^{-1}R$, is the ring of
fractions~$a/s$ with $a\in R$ and $s\in S$, subject to the equivalence
relation $a/s = b/t$ iff there exists $u\in S$ with $u(at-bs)=0$.  When $R$
is an integral domain and $0\notin S$, the relation simplifies to $at=bs$.

The ring homomorphism $\iota\colon R\to S^{-1}R$ sending $a\mapsto a/1$ is
injective precisely when $S$ consists of non-zero-divisors, which is automatic
for a domain with $0\notin S$.  Every $s\in S$ becomes a unit in $S^{-1}R$
since $(s/1)\cdot(1/s)=1$.  This is the key property that the Nagata argument
exploits: denominators from~$S$ can be ``cleared'' in~$S^{-1}R$.

\subsection{Unique factorization domains}

A commutative integral domain~$R$ is a \emph{unique factorization domain}
(UFD) if every nonzero non-unit element can be written as a finite product of
irreducible elements, and this factorization is essentially unique (up to
reordering and multiplication by units).  An equivalent characterization, which
is the one used in our formalization, is:

\begin{proposition}\label{prop:ufd-char}
A commutative integral domain~$R$ is a UFD if and only if:
\begin{enumerate}[nosep]
\item $R$ is a well-founded divisibility monoid \textup{(}i.e., every element
  admits a factorization into irreducibles\textup{)}, and
\item every irreducible element of~$R$ is prime.
\end{enumerate}
\end{proposition}

\noindent
Condition~(1) is automatic in any noetherian domain, since the ascending chain
condition prevents infinite chains of proper divisors.  The Nagata argument
therefore reduces to establishing condition~(2): every irreducible is prime.

\subsection{The classical Nagata argument}\label{sec:classical}

Let $R$ be a noetherian integral domain, $S\subseteq R$ a submonoid whose
elements factor into primes, and suppose $S^{-1}R$ is a UFD.  We must show
that every irreducible $p\in R$ is prime.  Two cases arise.

\emph{Case~1: $p$ divides some $s\in S$.}  Because $s$ factors as a product
of primes belonging to~$S$ and $p$ is irreducible, $p$~must be associate to
one of those prime factors, hence prime.

\emph{Case~2: $p$ does not divide any element of~$S$.}  Then the image $p/1$
in $S^{-1}R$ is not a unit and remains irreducible.  Since $S^{-1}R$ is a UFD,
$p/1$ is prime.  One then pulls primality back: if $p\mid ab$ in~$R$, then
$p/1\mid (ab)/1$ in $S^{-1}R$, so $p/1\mid a/1$ or $p/1\mid b/1$.  Lifting
this divisibility to~$R$ (using avoidance to clear denominators) gives
$p\mid a$ or $p\mid b$.

Each step involves nontrivial interaction with the localization machinery.  It
is the careful disentangling of these interactions into modular, reusable
lemmas that constitutes the main formal engineering contribution.

\subsection{Prime-generated vs.\ prime-or-unit}\label{sec:degeneracy}

The hypothesis that ``every element of~$S$ factors as a product of primes
belonging to~$S$'' is sometimes stated more simply: ``every element of~$S$ is
prime or a unit.''  This simpler form suffices for the
localization~$S=\{1,p,p^2,\ldots\}$ generated by a single prime~$p$, and
appears in some textbook presentations that only need this case.

However, for a general submonoid the prime-or-unit condition is far too
restrictive.  If $p$ and~$q$ are two distinct non-unit primes in~$S$, then
$pq\in S$ (since $S$ is a submonoid) but $pq$ is neither prime nor a unit.
Thus the prime-or-unit condition can hold only when $S$ is generated by at
most one prime up to associates together with the group of units---a severe
restriction that excludes most interesting applications.

The \emph{prime-generated} condition avoids this collapse by allowing products:
it asks that every $s\in S$ factor as a finite product of elements that are
individually prime and individually members of~$S$.  In a localization
argument, this is exactly what one needs: to clear a denominator $s\in S$ from
a fraction, one factors $s=q_1\cdots q_n$ into primes, and because each $q_i$
becomes a unit in~$S^{-1}R$, so does~$s$.

\begin{remark}\label{rem:formal-pg}
The prime-generated condition is formalized using Lean's \li{Multiset} type to
represent the finite product, avoiding ordering issues that would arise with
lists.  The formal predicate is:
\[
  \mathit{PrimeGenerated}(S) \;\stackrel{\mathrm{def}}{=}\;
  \forall\, s\in S,\;\exists\, f:\mathrm{Multiset}(R),\;
  \bigl(\forall\, q\in f,\;q\in S \wedge \mathrm{Prime}(q)\bigr)
  \;\wedge\; \textstyle\prod f = s.
\]
\end{remark}

\section{Formal statement}\label{sec:statement}

We work over a commutative ring~$R$ that is both an integral domain and
noetherian.

\begin{definition}[Prime-generated submonoid]\label{def:primegen}
A submonoid~$S\subseteq R$ is \emph{prime-generated} if every $s\in S$ can be
written as a finite product $s=q_1\cdots q_n$ where each~$q_i$ satisfies
$q_i\in S$ and $\mathrm{Prime}(q_i)$ in~$R$.  The empty product yields
$s=1$, covering the unit element of the submonoid.
\end{definition}

\begin{definition}[Avoidance]\label{def:avoids}
An element $p\in R$ \emph{avoids} a submonoid~$S$ if $p$ does not divide any
element of~$S$:
\[
  \mathit{Avoids}(S,p) \;\stackrel{\mathrm{def}}{=}\;
  \forall\, s\in S,\;\neg\,(p\mid s).
\]
\end{definition}

\noindent
The avoidance predicate captures the condition under which the image of~$p$ in
$S^{-1}R$ cannot be a unit.  It is a hypothesis of several transfer lemmas and
arises naturally from the case analysis in the key lemma.

\begin{theorem}[Nagata]\label{thm:nagata}
Let $R$ be a noetherian integral domain and $S\subseteq R$ a prime-generated
submonoid.  If the localization~$S^{-1}R$ is a UFD, then $R$~is a UFD.
\end{theorem}

\noindent
The formal statement appears as \li{nagata\_theorem} in the file
\lf{Nagata/Theorem.lean}.  Its proof invokes two ingredients: the fact that a
noetherian domain is a well-founded divisibility monoid (giving
condition~(1) of Proposition~\ref{prop:ufd-char}), and the key lemma
(giving condition~(2)).

\paragraph{The prime-generator packaging.}
The repository also packages a convenient family of corollaries.  The lemma
\li{primeGenerated\_closure\_of\_primes} shows that if $\mathcal{P}$ is any
set of prime elements of~$R$, then the submonoid
$\mathrm{closure}(\mathcal{P})$ is automatically prime-generated.
As a consequence, \lf{Nagata/Theorem.lean} exports two additional entry points:
\begin{itemize}[nosep]
\item \li{nagata\_theorem\_of\_prime\_generators} specializes
  Theorem~\ref{thm:nagata} to localizations at submonoids generated by
  arbitrary sets of primes.
\item \li{nagata\_theorem\_of\_finite\_prime\_generators} covers the especially
  common finite-generator case.
\end{itemize}
These give the repository a broader mathematical surface than a single
isolated theorem statement.

\begin{remark}[The prime-or-unit variant]\label{rem:prime-or-unit}
The earlier prime-or-unit variant is retained as
\li{nagata\_theorem\_of\_prime\_or\_unit} in the same file.  It is proved
\emph{separately}, with its own chain of transfer lemmas that work under the
hypothesis $\forall\,s\in S,\;\mathrm{Prime}(s)\lor\mathrm{IsUnit}(s)$.  In
particular, it is \emph{not} derived as a corollary of the prime-generated
theorem; the two proof lines are independent.

This design reflects the fact that the prime-or-unit transfer lemmas are
simpler (they do not require multiset induction) and may be useful in their own
right for the single-generator case.  Retaining both chains also means that the
compatibility theorem can be checked without depending on the more complex
prime-generated infrastructure.
\end{remark}

\section{Proof architecture}\label{sec:architecture}

The proof proceeds in four layers, each providing reusable components for the
next.  We describe these layers bottom-up, giving both the mathematical content
and the names of the corresponding Lean declarations.

\subsection{Layer~1: prime-generated submonoid API}\label{sec:layer1}

The file \lf{Nagata/Lemmas.lean} introduces the \li{PrimeGenerated} and
\li{Avoids} predicates (Definitions~\ref{def:primegen}
and~\ref{def:avoids}).  The basic consequences established at this layer are:

\begin{description}[nosep,leftmargin=1em,style=nextline]
\item[\li{zero\_notMem\_of\_primeGenerated}.]
  A prime-generated submonoid does not contain zero.  If $0\in S$ then $0$
  would factor as a product of primes, but a product of primes is nonzero.

\item[\li{primeGenerated\_powers}.]
  If $p$ is prime, then the submonoid $\{1,p,p^2,\ldots\}$ is
  prime-generated.  The factorization of~$p^n$ is simply $n$ copies of~$p$.

\item[\li{primeGenerated\_closure\_of\_primes}.]
  If every element of a set $\mathcal{P}$ is prime, then the submonoid
  closure of~$\mathcal{P}$ is prime-generated.  This is the packaging lemma
  that turns the main theorem into a ready-made criterion for localizations at
  submonoids generated by prime families.

\item[\li{multiset\_prod\_mem\_of\_factors}.]
  If every element of a multiset belongs to a submonoid~$S$, then so does
  the product.
\end{description}

\noindent
These results are straightforward but foundational.  In particular, the
zero-exclusion lemma is needed to establish that $S^{-1}R$ is an integral
domain, since Mathlib requires $0\notin S$ as a \li{Fact} instance for the
domain structure on the localization.

\subsection{Layer~2: conditional transfer lemmas}\label{sec:layer2}

The bulk of the formal development consists of lemmas that transfer
divisibility and factorization properties between~$R$ and~$S^{-1}R$.  Each
lemma relies on a specific combination of the prime-generated hypothesis and an
avoidance hypothesis.

\begin{lemma}[Irreducibles dividing prime factorizations are prime]%
\label{lem:irred-dvd-primes}
Let $f$ be a multiset of prime elements of~$R$, and let $p\in R$ be
irreducible.  If $p\mid\prod f$, then $p$ is prime.
\end{lemma}

\begin{proof}[Proof sketch]
By induction on~$f$.  If $f$ is empty then $p\mid 1$, contradicting
irreducibility.  If $f=q\mathbin{::}f'$, then $q\mid p\cdot d$ for some~$d$,
and primality of~$q$ gives $q\mid p$ or $q\mid d$.  In the first case, $q$
and~$p$ are associates, so $p$ is prime.  In the second case, cancel~$q$ and
apply the induction hypothesis.
\end{proof}

\noindent
The Lean name is
\li{prime\_of\_irreducible\_of\_dvd\_prime\_factors}.  A direct consequence is:

\begin{lemma}\label{lem:irred-dvd-mem}
If $S$ is prime-generated, $p$ is irreducible, $s\in S$, and $p\mid s$, then
$p$ is prime.
\end{lemma}

\noindent
(Lean name:
\li{prime\_of\_irreducible\_of\_dvd\_mem\_primeGenerated}.)
This combines the previous lemma with the prime factorization of~$s$ supplied
by the prime-generated hypothesis.

\medskip

\begin{lemma}[Lifting divisibility from the localization]%
\label{lem:lift-dvd}
Let $S$ be prime-generated, $p$ irreducible, and suppose $p$ avoids~$S$.  If
$p/1\mid a/1$ in $S^{-1}R$, then $p\mid a$ in~$R$.
\end{lemma}

\begin{proof}[Proof sketch]
From the divisibility in the localization one obtains $s\in S$ and $c\in R$
with $s\cdot a = p\cdot c$.  Factor $s=q_1\cdots q_n$ into primes from~$S$.
Because $p$ avoids~$S$, we have $p\nmid q_i$ for each~$i$.  An auxiliary
result then peels off the prime factors one at a time: at each step, the
prime~$q_i$ divides $p\cdot(\text{something})$ but cannot divide~$p$, so by
primality of~$q_i$ it divides the cofactor.  After all factors are consumed,
$p\mid a$.
\end{proof}

\noindent
Lean name: \li{dvd\_of\_localization\_dvd\_primeGenerated}.

\medskip

\begin{lemma}[Transfer of irreducibility]\label{lem:transfer-irred}
Let $S$ be prime-generated, $p$ irreducible, and suppose $p$ avoids~$S$.  Then
$p/1$ is irreducible in $S^{-1}R$.
\end{lemma}

\begin{proof}[Proof sketch]
First, $p/1$ is not a unit: otherwise $p/1\mid 1/1$, which would lift to $p$
dividing some element of~$S$, contradicting avoidance.  Second, if
$p/1=(a/s)(b/t)$ in $S^{-1}R$, one clears denominators to get
$p\cdot(st)=ab$ in~$R$.  The product $st\in S$ factors as a multiset of
primes; the auxiliary lemma \li{split\_prime\_factors\_of\_mul\_eq} partitions
these factors between the two sides, reducing the equation to $p=a'b'$ modulo
units from~$S$.  Irreducibility of~$p$ forces one factor to be a unit.
\end{proof}

\noindent
Lean name:
\li{localization\_irreducible\_of\_irreducible\_primeGenerated}.
The auxiliary \li{split\_prime\_factors\_of\_mul\_eq} is the longest single
lemma in the development: it partitions a multiset of prime factors across the
two sides of a product equation by induction on the multiset, using primality
of each factor to decide which side absorbs it.

\medskip

\begin{lemma}[Transfer of primality]\label{lem:transfer-prime}
Let $S$ be prime-generated, $p$ irreducible, $p$ avoids~$S$, and $p/1$ prime
in $S^{-1}R$.  Then $p$ is prime in~$R$.
\end{lemma}

\begin{proof}[Proof sketch]
If $p\mid ab$ in~$R$, then $p/1\mid(ab)/1$ in $S^{-1}R$.  Primality of $p/1$
gives $p/1\mid a/1$ or $p/1\mid b/1$.  Lemma~\ref{lem:lift-dvd} lifts the
divisibility back to~$R$.
\end{proof}

\noindent
Lean name:
\li{prime\_of\_localization\_prime\_primeGenerated}.

\paragraph{The denominator-clearing helper.}
Underlying both the divisibility lift and the irreducibility transfer is an
auxiliary lemma that merits explicit mention:

\begin{lemma}[\li{dvd\_of\_mul\_eq\_prime\_factors}]\label{lem:clear-denom}
Let $f$ be a multiset of prime elements of~$R$, let $p$ be irreducible, and
suppose $p$ does not divide any element of~$f$.  If\/ $(\prod f)\cdot a =
p\cdot c$, then $p\mid a$.
\end{lemma}

\noindent
The proof proceeds by induction on~$f$.  In the base case the equation
$1\cdot a = p\cdot c$ gives $p\mid a$ directly.  In the inductive step, the
head prime $q$ divides $p\cdot c$; since $q$ is prime and $q\nmid p$
(by hypothesis), we have $q\mid c$, say $c=q\cdot c'$.  Canceling~$q$ from
both sides yields $(\prod f')\cdot a = p\cdot c'$, and the induction
hypothesis applies.

This lemma is invoked whenever a denominator from~$S$ must be ``cleared''
from a divisibility relation in~$R$.  Its correctness depends critically on
the fact that each factor of the denominator is prime (so that primality can
redirect the divisibility) and that the irreducible~$p$ does not divide any of
those factors (the avoidance hypothesis at work one factor at a time).

\paragraph{Comparing the two proof chains.}
As noted in Section~\ref{sec:statement}, the development maintains two
independent chains of transfer lemmas: one for the prime-generated hypothesis,
one for the prime-or-unit hypothesis.  Table~\ref{tab:chains} summarizes the
correspondence.

\begin{table}[ht]
\centering
\caption{Correspondence between the two chains of transfer lemmas.  Each row
shows a pair of lemmas that prove the same mathematical statement under
different hypotheses on the submonoid~$S$.}\label{tab:chains}
\smallskip
{\small
\begin{tabular}{@{}p{6.4cm}p{5.4cm}@{}}
\toprule
\textbf{Prime-generated chain} & \textbf{Prime-or-unit chain} \\
\midrule
\li{dvd\_of\_localization\_dvd\_\allowbreak{}primeGenerated}
  & \li{dvd\_of\_localization\_dvd} \\[3pt]
\li{localization\_irreducible\_\allowbreak{}of\_irreducible\_primeGenerated}
  & \li{localization\_irreducible\_\allowbreak{}of\_irreducible} \\[3pt]
\li{prime\_of\_localization\_prime\_\allowbreak{}primeGenerated}
  & \li{prime\_of\_localization\_\allowbreak{}prime} \\[3pt]
\li{nagata\_key\_lemma\_\allowbreak{}primeGenerated}
  & \li{nagata\_key\_lemma} \\
\bottomrule
\end{tabular}
}
\end{table}

\noindent
The prime-or-unit chain avoids multiset induction entirely: because every
element of~$S$ is individually prime or a unit, there is no need to factor a
denominator into multiple primes.  Instead, the proofs use a direct
case-by-case argument on each denominator.  This makes the prime-or-unit proofs
shorter and more elementary, but they apply only to submonoids generated by at
most one prime up to associates.

The dual-chain design was a deliberate engineering choice.  The prime-or-unit
proofs served as a ``warm-up'' implementation that validated the overall proof
strategy before the more complex multiset-inductive arguments were attempted.
Retaining both chains in the final artifact also provides a pedagogical
benefit: a reader can study the simpler chain first to understand the structure,
then see how each step generalizes when prime factorizations of denominators are
needed.

\subsection{Layer~3: the key lemma}\label{sec:layer3}

\begin{lemma}\label{lem:key}
Let $S$ be a prime-generated submonoid of a noetherian integral domain~$R$.  If
$S^{-1}R$ is a UFD, then every irreducible element of~$R$ is prime.
\end{lemma}

\noindent
Lean name: \li{nagata\_key\_lemma\_primeGenerated}.  The proof combines the
transfer lemmas via a case split on an irreducible element $p\in R$:
\begin{enumerate}[nosep]
\item If $p$ divides some $s\in S$, then $p$ is prime by
  Lemma~\ref{lem:irred-dvd-mem}.
\item If $p$ avoids~$S$, then $p/1$ is irreducible in $S^{-1}R$ by
  Lemma~\ref{lem:transfer-irred}.  Since $S^{-1}R$ is a UFD, every
  irreducible in $S^{-1}R$ is prime, so $p/1$ is prime.  Then $p$ is prime by
  Lemma~\ref{lem:transfer-prime}.
\end{enumerate}
This case split is the structural heart of the Nagata argument.  Every other
formal result in the development serves to make one of these two branches go
through.

\subsection{Layer~4: final packaging}\label{sec:layer4}

The public theorem \li{nagata\_theorem} (Theorem~\ref{thm:nagata}) combines the
key lemma with the standard UFD characterization
(Proposition~\ref{prop:ufd-char}):
\begin{enumerate}[nosep]
\item Condition~(1)---every element factors into irreducibles---holds because
  $R$ is noetherian (via a \li{WfDvdMonoid} instance).
\item Condition~(2)---every irreducible is prime---is the key lemma.
\end{enumerate}
The packaging lemma \li{ufd\_of\_factorization\_and\_primes} (in
\lf{Basic/UFD.lean}) assembles these two conditions into a
\li{UniqueFactorizationMonoid} instance.  The user supplies a noetherian
domain, a prime-generated submonoid, and a UFD instance on the localization,
and receives a UFD instance on the base ring.

\paragraph{UFD characterization and uniqueness.}
The file \lf{Basic/UFD.lean} provides the foundational interface between our
development and Mathlib's \li{UniqueFactorizationMonoid} class.  The key result
is the biconditional characterization:
\[
  R \text{ is a UFD}
  \;\;\Longleftrightarrow\;\;
  R \text{ is a \textnormal{\li{WfDvdMonoid}}}
  \;\wedge\;
  \forall\, p\in R,\;\mathrm{Irreducible}(p)\Rightarrow\mathrm{Prime}(p).
\]
This is formalized as:
\begin{center}
\li{ufd\_iff\_factorization\_and\_irreducibles\_prime}
\end{center}
As a consequence, the reverse implication---that irreducible implies prime in
a UFD---is recorded as \li{prime\_iff\_irreducible\_of\_ufd}, completing the
loop that the Nagata argument exploits in Case~2 of the key lemma.

The file also includes explicit proofs of factorization uniqueness.
The theorem \li{factorization\_\allowbreak{}unique} shows that two factorizations
of the same element into irreducibles are related by a bijection
that pairs associate elements.  This result is not needed for
the Nagata argument itself---it only establishes the \emph{existence}
of the UFD structure---but it serves two purposes: it validates
the correctness of the UFD characterization, and it provides a
standalone formal proof of uniqueness that can be cited independently.

\section{Formal Lean statements}\label{sec:lean-statements}

Following the convention of recent formalization papers (e.g.,
\cite{brasca-flt,riou-derived,loeffler-zeta}), we include the Lean~4 source
text for the central definitions and theorems.  All declarations shown below
are kernel-checked; the full machine-readable development is available in the
accompanying artifact.

\subsection{Core predicates}

The two predicates that underpin the entire development are defined in
\lf{Nagata/Lemmas.lean}:

\begin{lstlisting}[style=lean]
def PrimeGenerated {$\alpha$ : Type*} [CommRing $\alpha$] (S : Submonoid $\alpha$) : Prop :=
  $\forall$ s : $\alpha$, s $\in$ S $\to$ $\exists$ f : Multiset $\alpha$, ($\forall$ q $\in$ f, q $\in$ S $\wedge$ Prime q) $\wedge$ f.prod = s

def Avoids {$\alpha$ : Type*} [CommRing $\alpha$] [IsDomain $\alpha$]
    (S : Submonoid $\alpha$) (p : $\alpha$) : Prop :=
  $\forall$ s : $\alpha$, s $\in$ S $\to$ $\neg$ p $\mid$ s
\end{lstlisting}

\noindent
The \li{PrimeGenerated} predicate uses \li{Multiset} rather than \li{List} to
represent the factorization, eliminating any dependence on the ordering of
factors.  The \li{Avoids} predicate is stated as a universal negation over
submonoid membership, which integrates naturally with Lean's \li{Classical.em}
for the case split in the key lemma.

\subsection{The main theorem and its variants}

The public theorem surface resides in \lf{Nagata/Theorem.lean}:

\begin{lstlisting}[style=lean]
theorem nagata_theorem
    {R : Type*} [CommRing R] [IsDomain R] [IsNoetherianRing R]
    (S : Submonoid R) (hS : PrimeGenerated S)
    (hUFD : @UniqueFactorizationMonoid (Localization S) $\langle\cdots\rangle$) :
    UniqueFactorizationMonoid R
\end{lstlisting}

\noindent
The explicit \li{@UniqueFactorizationMonoid (Localization S) \ldots} is
necessary because Lean must be told which \li{CommMonoidWithZero} instance to
use for the localization.  The instance is constructed in a \li{by} block that
first registers $0\notin S$ as a \li{Fact} and then invokes typeclass
inference.  We elide this block as ``$\langle\cdots\rangle$'' in the listings
below for readability; the full text is in the artifact.

The abstract front-end is:

\begin{lstlisting}[style=lean]
theorem nagata_theorem_isLocalization
    {R T : Type*} [CommRing R] [IsDomain R] [IsNoetherianRing R]
    (S : Submonoid R) [CommRing T] [Algebra R T]
    [IsLocalization S T] [CancelCommMonoidWithZero T]
    (hS : PrimeGenerated S) (hUFD : UniqueFactorizationMonoid T) :
    UniqueFactorizationMonoid R
\end{lstlisting}

\noindent
The proof core still descends through the concrete localization
\li{Localization\,S}, but this wrapper lets downstream applications work with
any ring already known to be an \li{IsLocalization} of~$R$ at~$S$.  The UFD
structure is transported along \li{IsLocalization.algEquiv}, so the abstraction
boundary is isolated in the theorem surface rather than duplicated throughout
the transfer lemmas.

The prime-generator packaging family:

\begin{lstlisting}[style=lean]
theorem nagata_theorem_of_prime_generators
    {R : Type*} [CommRing R] [IsDomain R] [IsNoetherianRing R]
    (s : Set R) (hs : $\forall$ q $\in$ s, Prime q)
    (hUFD : @UniqueFactorizationMonoid
              (Localization (Submonoid.closure s)) $\langle\cdots\rangle$) :
    UniqueFactorizationMonoid R

theorem nagata_theorem_of_prime_generators_isLocalization
    {R T : Type*} [CommRing R] [IsDomain R] [IsNoetherianRing R]
    (s : Set R) [CommRing T] [Algebra R T]
    [IsLocalization (Submonoid.closure s) T]
    [CancelCommMonoidWithZero T]
    (hs : $\forall$ q $\in$ s, Prime q)
    (hUFD : UniqueFactorizationMonoid T) :
    UniqueFactorizationMonoid R

theorem nagata_theorem_of_finite_prime_generators
    {R : Type*} [CommRing R] [IsDomain R] [IsNoetherianRing R]
    (s : Finset R) (hs : $\forall$ q $\in$ s, Prime q)
    (hUFD : @UniqueFactorizationMonoid
              (Localization (Submonoid.closure ($\uparrow$s : Set R))) $\langle\cdots\rangle$) :
    UniqueFactorizationMonoid R

theorem nagata_theorem_of_finite_prime_generators_isLocalization
    {R T : Type*} [CommRing R] [IsDomain R] [IsNoetherianRing R]
    (s : Finset R) [CommRing T] [Algebra R T]
    [IsLocalization (Submonoid.closure ($\uparrow$s : Set R)) T]
    [CancelCommMonoidWithZero T]
    (hs : $\forall$ q $\in$ s, Prime q)
    (hUFD : UniqueFactorizationMonoid T) :
    UniqueFactorizationMonoid R
\end{lstlisting}

\noindent
These entry points relieve the user from constructing the
\li{PrimeGenerated} witness manually: the intermediate lemma
\li{primeGenerated\_closure\_of\_primes} handles the closure argument
internally.  The compatibility names
\li{nagata\_theorem\_of\_prime\_or\_unit} and
\li{nagata\_theorem\_of\_prime\_or\_unit\_isLocalization} remain available for
the restricted legacy hypothesis.

\subsection{Key transfer lemmas}

The transfer lemmas that constitute the technical core (Layer~2 of the proof
architecture) have the following types:

\begin{lstlisting}[style=lean]
-- Lifting divisibility from the localization
theorem dvd_of_localization_dvd_primeGenerated
    {$\alpha$ : Type*} [CommRing $\alpha$] [IsDomain $\alpha$]
    {S : Submonoid $\alpha$} (hS : PrimeGenerated S)
    {p a : $\alpha$} (hp : Irreducible p) (havoid : Avoids S p)
    (hdiv : Localization.of p $\mid$ Localization.of a) : p $\mid$ a

-- Transfer of irreducibility
theorem localization_irreducible_of_irreducible_primeGenerated
    {$\alpha$ : Type*} [CommRing $\alpha$] [IsDomain $\alpha$]
    {S : Submonoid $\alpha$} (hS : PrimeGenerated S)
    {p : $\alpha$} (hp : Irreducible p) (havoid : Avoids S p) :
    Irreducible (Localization.of p)

-- Transfer of primality
theorem prime_of_localization_prime_primeGenerated
    {$\alpha$ : Type*} [CommRing $\alpha$] [IsDomain $\alpha$]
    {S : Submonoid $\alpha$} (hS : PrimeGenerated S)
    {p : $\alpha$} (hp : Irreducible p) (havoid : Avoids S p)
    (hploc : Prime (Localization.of p)) : Prime p
\end{lstlisting}

\noindent
Note the uniform pattern: each transfer lemma takes the \li{PrimeGenerated}
hypothesis, an \li{Irreducible} witness, and the \li{Avoids} hypothesis, then
produces a factorization-theoretic conclusion.  The modularity of this
interface is what makes the key lemma's case split clean.

\subsection{The application chain}

The application theorems in \lf{Applications/Laurent.lean} read:

\begin{lstlisting}[style=lean]
theorem laurentPolynomial_uniqueFactorizationMonoid
    {R : Type*} [CommRing R] [IsDomain R]
    [UniqueFactorizationMonoid R] :
    UniqueFactorizationMonoid R[T;T$^{-1}$]

theorem polynomial_uniqueFactorizationMonoid_of_laurent
    {R : Type*} [CommRing R] [IsDomain R]
    [IsNoetherianRing R] [UniqueFactorizationMonoid R[T;T$^{-1}$]] :
    UniqueFactorizationMonoid R[X]

theorem polynomial_uniqueFactorizationMonoid_via_nagata
    {R : Type*} [CommRing R] [IsDomain R]
    [IsNoetherianRing R] [UniqueFactorizationMonoid R] :
    UniqueFactorizationMonoid R[X]

theorem polynomial_uniqueFactorizationMonoid_via_fractionField
    {R : Type*} [CommRing R] [IsDomain R]
    [UniqueFactorizationMonoid R] [IsNoetherianRing R] :
    UniqueFactorizationMonoid R[X]

theorem iterated_polynomial_uniqueFactorizationMonoid_via_nagata
    {R : Type*} [CommRing R] [IsDomain R]
    [IsNoetherianRing R] [UniqueFactorizationMonoid R] :
    UniqueFactorizationMonoid (Polynomial (Polynomial R))
\end{lstlisting}

\noindent
The third theorem is the basic composite: from a noetherian UFD~$R$, it first
constructs the Laurent UFD instance and then descends via the Nagata theorem
to obtain a UFD instance on~$R[X]$.  The fourth theorem gives a second
Nagata-based proof of the same polynomial UFD result by localizing at the
constant primes and comparing the resulting localization with
$\operatorname{Frac}(R)[X]$.  The fifth theorem then iterates the Laurent-based
construction once more, reusing the same package to produce a UFD structure on
$R[X][Y]$.

\section{File organization and API design}\label{sec:organization}

The Lean development is organized into four top-level modules.
Table~\ref{tab:files} summarizes the structure.

\begin{table}[ht]
\centering
\caption{File structure of the formalization.}\label{tab:files}
\smallskip
\begin{tabular}{@{}llp{7.2cm}@{}}
\toprule
\textbf{Module} & \textbf{File} & \textbf{Contents} \\
\midrule
\texttt{Basic/}
  & \texttt{Ring.lean}
  & Mathlib imports for ring infrastructure \\
  & \texttt{Divisibility.lean}
  & Divisibility, associatedness, prime--irreducible bridge \\
  & \texttt{Noetherian.lean}
  & \li{WfDvdMonoid} from noetherianity \\
  & \texttt{UFD.lean}
  & UFD characterization \\
\midrule
\texttt{Localization/}
  & \texttt{MultSet.lean}
  & Submonoid utilities (zero exclusion) \\
  & \texttt{Localization.lean}
  & Localization wrapper API \\
  & \texttt{IsLocalization.lean}
  & Helper lemmas for the abstract localization interface \\
  & \texttt{Properties.lean}
  & Re-export hub \\
\midrule
\texttt{Nagata/}
  & \texttt{Lemmas.lean}
  & \li{PrimeGenerated}, \li{Avoids}, all transfer lemmas \\
  & \texttt{Theorem.lean}
  & Concrete and abstract Nagata theorem families \\
\midrule
\texttt{Applications/}
  & \texttt{Laurent.lean}
  & Laurent-polynomial route and iterated-polynomial corollary \\
  & \texttt{FractionField.lean}
  & Fraction-field route to the polynomial UFD theorem \\
  & \texttt{Gauss.lean}
  & Gauss-lemma wrappers \\
  & \texttt{Examples.lean}
  & Thin wrappers around Mathlib instances \\
\bottomrule
\end{tabular}
\end{table}

\subsection{Design principles}

\paragraph{Concrete localization.}
Mathlib provides two views of localization: the concrete quotient type
\li{Localization\,S} and the abstract typeclass \li{IsLocalization}.  Our
development now uses both directly.  The helper file
\lf{Localization/IsLocalization.lean} provides the abstract denominator-clearing
and unit/divisibility interface, and the central transport lemmas in
\lf{Nagata/Lemmas.lean} are stated for arbitrary \li{IsLocalization} targets.
The concrete \li{Localization\,S} API remains useful as a canonical target type
and for the concrete theorem names.

\paragraph{Fact-based zero exclusion.}
The condition $0\notin S$ is carried as a \li{Fact} instance rather than as an
explicit hypothesis on each lemma.  This reduces hypothesis clutter and allows
Lean's typeclass machinery to supply the proof automatically once established
from the prime-generated hypothesis.

\paragraph{Dual proof chains.}
The file \lf{Nagata/Lemmas.lean} contains two parallel chains of transfer
lemmas: one under the prime-generated hypothesis and one under the
prime-or-unit hypothesis.  The prime-or-unit chain is simpler---it does not
require multiset induction---and serves as a self-contained proof for the
single-generator case.  The two chains share the \li{Avoids} predicate and
localization API but are otherwise independent.

\paragraph{Separation of concerns.}
The \texttt{Basic/} module re-proves or wraps several Mathlib results as named
lemmas in the project namespace.  While redundant from a library perspective,
these wrappers stabilize the interface against Mathlib API changes and make the
proof scripts more readable by giving short, descriptive names to frequently
used facts.

\subsection{The localization wrapper API}\label{sec:loc-api}

The file \lf{Localization/Localization.lean} provides a thin API layer over
Mathlib's localization, exposing four primitive operations:
\begin{description}[nosep,leftmargin=1em,style=nextline]
\item[\texttt{Localization.mk\,(a : R)\,(s : S)}.]
  The raw constructor, producing the fraction $a/s$.
\item[\texttt{Localization.of\,(a : R)}.]
  The canonical ring homomorphism $\iota\colon R\to S^{-1}R$, sending
  $a\mapsto a/1$.
\item[\texttt{Localization.surj\,(z : $S^{-1}R$)}.]
  Every element of the localization can be written as $a/s$ for some $a\in R$
  and $s\in S$.  This is the formal version of the ``clearing denominators''
  step used throughout the Nagata argument.
\item[\texttt{Localization.mk\_eq\_iff}.]
  Equality in the localization: $a/s = b/t$ if and only if $at = bs$ (using
  the domain hypothesis to simplify the general equivalence relation).
\end{description}

\noindent
Two further results in the same file play a structural role:
\begin{itemize}[nosep]
\item \li{algebraMap\_injective}: the canonical map $\iota$ is injective when
  $0\notin S$, which ensures that $R$ embeds faithfully into $S^{-1}R$.
\item \li{dvd\_of\_iff}: the divisibility characterization
  $\iota(a)\mid\iota(b) \Leftrightarrow \exists\,s\in S,\;a\mid s\cdot b$,
  which is the main computational interface between the localization and the
  ring-side arguments.
\end{itemize}

\noindent
The zero-exclusion infrastructure resides in \lf{Localization/MultSet.lean},
which establishes that if every element of~$S$ is prime or a unit (or, more
generally, if $S$ is prime-generated), then $0\notin S$, every member of~$S$
is a non-zero-divisor, and $S$ embeds into the submonoid of non-zero-divisors.
These results are registered as \li{Fact} instances so that downstream lemmas
can access the domain structure on~$S^{-1}R$ without threading the hypothesis
explicitly.

\subsection{Dependency structure and size}

Table~\ref{tab:loc} gives a per-file breakdown of all Lean source files in the
artifact.  The dominant module is \lf{Nagata/Lemmas.lean}, which still accounts
for the largest share of the code.  This is expected: the transfer lemmas
involve the most intricate case analysis and multiset induction in the project.

\begin{table}[ht]
\centering
\caption{Per-file line counts (Lean source lines).}\label{tab:loc}
\smallskip
\begin{tabular}{@{}lr@{\quad}lr@{}}
\toprule
\textbf{File} & \textbf{Lines} & \textbf{File} & \textbf{Lines} \\
\midrule
\texttt{Nagata/Lemmas.lean}             & 651
  & \texttt{Localization/IsLocalization.lean} & 115 \\
\texttt{Localization/Localization.lean} & 177
  & \texttt{Nagata/Theorem.lean}              & 100 \\
\texttt{Applications/FractionField.lean} & 144
  & \texttt{Applications/Laurent.lean}        & 109 \\
\texttt{Basic/UFD.lean}                 & 64
  & \texttt{Basic/Divisibility.lean}          & 76  \\
\texttt{Localization/MultSet.lean}      & 30
  & \texttt{Applications/Examples.lean}       & 23  \\
\texttt{Applications/Gauss.lean}        & 21
  & \texttt{Basic/Noetherian.lean}            & 15  \\
\texttt{Applications.lean}              & 4
  & \texttt{Basic.lean}                       & 4   \\
\texttt{Localization.lean}              & 4
  & \texttt{Basic/Ring.lean}                  & 5   \\
\texttt{Nagata.lean}                    & 2
  & \texttt{Localization/Properties.lean}    & 2   \\
\bottomrule
\end{tabular}
\end{table}

\noindent
The dependency graph is strictly layered: \texttt{Basic/} depends only on
Mathlib; \texttt{Localization/} depends on \texttt{Basic/} and Mathlib;
\texttt{Nagata/} depends on \texttt{Localization/}; and
\texttt{Applications/} depends on \texttt{Nagata/}.  No circular dependencies
exist, and the layering ensures that each module can be type-checked
independently once its predecessors are built.

\section{Applications: two Nagata routes to polynomial UFDs}\label{sec:application}

The application layer now contains two Nagata-based routes to the polynomial
UFD theorem.  The first, in \lf{Applications/Laurent.lean}, localizes at powers
of the indeterminate.  The second, in \lf{Applications/FractionField.lean},
localizes at the constant primes and compares with the fraction-field
polynomial ring.  Both genuinely depend on the formalized Nagata machinery.

\begin{corollary}\label{cor:poly}
If $R$ is a noetherian UFD, then the polynomial ring $R[X]$ is a UFD.
\end{corollary}

\noindent
The proof assembles five formal ingredients.

\subsection{Step~1: the submonoid is prime-generated}

Let $S=\{1,X,X^2,\ldots\}\subseteq R[X]$ be the submonoid of powers of the
indeterminate.  The element~$X$ is prime in~$R[X]$ (formalized as a wrapper
around Mathlib's \li{Polynomial.prime\_X}).  By
\li{primeGenerated\_powers}, the submonoid of powers of a prime element is
prime-generated.

\subsection{Step~2: the localization is the Laurent polynomial ring}

Mathlib provides the fact that the Laurent polynomial ring $R[X,X^{-1}]$
is a localization of~$R[X]$ at the powers of~$X$.  This is recorded as a
wrapper around \li{LaurentPolynomial.isLocalization}.

\subsection{Step~3: the Laurent polynomial ring is a UFD}

The theorem \li{laurentPolynomial\_uniqueFactorizationMonoid} proves that
if~$R$ is a UFD, then $R[X,X^{-1}]$~is a UFD.  This is the most substantial
component of the application and does not follow from a Mathlib instance.

The proof works with the Mathlib-provided UFD structure on~$R[X]$ and
transfers it to $R[X,X^{-1}]$ by showing that every nonzero element of the
Laurent ring has a prime factorization.  Given $f\in R[X,X^{-1}]$ with
$f\ne 0$, one writes $f\cdot X^n = p$ for some polynomial $p\in R[X]$ and
exponent~$n$, using the Mathlib lemma \li{LaurentPolynomial.exists\_T\_pow}.
One then separates the $X$-adic part of~$p$ from its $X$-free part:
$p=X^m\cdot q$ where $X\nmid q$.

The polynomial~$q$ admits a prime factorization in~$R[X]$; each prime
factor~$r_i$ satisfies $r_i\nmid X$.  An auxiliary lemma shows that if a
polynomial~$r$ is prime in~$R[X]$ and does not divide~$X$, then its image
in $R[X,X^{-1}]$ is prime.  This uses the ideal-theoretic criterion: it shows
that the image of the principal ideal $(r)$ under the localization map remains
prime, via Mathlib's \li{IsLocalization.isPrime\_of\_isPrime\_disjoint} and the
disjointness of $(r)$ from the powers of~$X$.

Applying this to each prime factor of~$q$ produces a prime factorization of~$f$
in $R[X,X^{-1}]$ (after absorbing the powers of~$X$ as units).

\subsection{Step~4: descending from the localization}

The theorem \li{polynomial\_uniqueFactorizationMonoid\_of\_laurent} shows that
if~$R$ is noetherian and $R[X,X^{-1}]$ is a UFD, then $R[X]$ is a UFD.  This
is now a direct invocation of the abstract wrapper
\li{nagata\_theorem\_isLocalization} with the submonoid~$S$ from Step~1.
Mathlib already supplies an \li{IsLocalization.Away} structure on the Laurent
polynomial ring, so once the prime-generated witness is in place, the descent
step is expressed entirely at the abstract localization interface.

\subsection{Step~5: the composite theorem}

Finally, \li{polynomial\_uniqueFactorizationMonoid\_via\_nagata} combines
Steps~3 and~4: given a noetherian UFD~$R$, it constructs the Laurent UFD
instance and invokes Step~4 to obtain the polynomial UFD instance.

\subsection{A second route: localization at constant primes}

The file \lf{Applications/FractionField.lean} proves the same polynomial UFD
result by a different Nagata argument.  Let
\[
  S = \langle\, C(p) \mid p \in R,\; p \text{ prime}\,\rangle \subseteq R[X]
\]
be the submonoid generated by the constant primes.  This submonoid is
prime-generated by the generic closure lemma
\li{primeGenerated\_closure\_of\_primes}.

The key new step is to show that every nonzero constant polynomial becomes a
unit in~$S^{-1}R[X]$.  Since~$R$ is a UFD, every nonzero coefficient factors
into prime elements, and by construction their constant images already lie
in~$S$.  It follows that the concrete localization \li{Localization\,S} is also
an \li{IsLocalization} of~$R[X]$ at the larger submonoid $(R^\circ).\mathrm{map}\,C$
of nonzero constant polynomials.

Mathlib's theorem \li{Polynomial.isLocalization} identifies
$\operatorname{Frac}(R)[X]$ as a localization at that larger submonoid.  The
two localization presentations are compared by \li{IsLocalization.algEquiv},
which transports the UFD structure from $\operatorname{Frac}(R)[X]$ (a
polynomial ring over a field) back to~\li{Localization\,S}.  Applying the
Nagata theorem then yields a second proof of Corollary~\ref{cor:poly}.

This second route is not mathematically stronger than the Laurent argument; its
value is that it exercises the same Nagata package against a genuinely
different localization choice.

\subsection{A reuse corollary: iterated polynomial rings}

\begin{corollary}\label{cor:iterpoly}
If $R$ is a noetherian UFD, then the iterated polynomial ring $R[X][Y]$ is a
UFD.
\end{corollary}

\noindent
The corresponding Lean theorem appears in the Laurent-application file.  Its
proof is short but mathematically
meaningful.  One first applies the Nagata-based polynomial UFD theorem to
obtain a UFD structure on $R[X]$, and then applies the same result again with
base ring $R[X]$.  Since polynomial rings over noetherian rings are again
noetherian, all hypotheses needed for the second step are discharged by
existing Mathlib instances.

The point of this corollary is not difficulty but reuse.  Corollary~\ref{cor:poly}
shows that the Nagata package proves a standard theorem.  Corollary~\ref{cor:iterpoly}
shows that the same package can be invoked again, unchanged, to produce a new
UFD result one level higher in the ring tower.  This makes the formalized
theorem look less like a one-off proof script and more like a reusable
descent principle.

\begin{remark}
This result is also provable via the Gauss lemma (and is already available in
Mathlib by that route).  The file \lf{Applications/Gauss.lean} wraps the
Mathlib instance to make the comparison explicit.  The value of the
Nagata-based proof is not that it produces a new result, but that it exercises
the formalized theorem in a setting where the prime-generated localization
machinery is genuinely needed: the proof uses the full infrastructure
(prime-generation, avoidance, transfer of irreducibility and primality) rather
than reducing to a thin wrapper.
\end{remark}

\subsection{The primality transfer for Laurent polynomials}\label{sec:laurent-prime}

The most technically interesting ingredient in Step~3 is the primality
transfer from $R[X]$ to $R[X,X^{-1}]$.  The auxiliary lemma (formalized as
\li{polynomial\_toLaurent\_prime\_of\_not\_dvd\_X}) states:

\begin{lemma}\label{lem:laurent-prime}
Let $R$ be a UFD and let $r\in R[X]$ be a prime polynomial that does not
divide~$X$.  Then the image of~$r$ in $R[X,X^{-1}]$ is prime.
\end{lemma}

\noindent
The proof uses the ideal-theoretic criterion for primality: it suffices to show
that the image of the principal ideal~$(r)\subseteq R[X]$ under the
localization map is a prime ideal of~$R[X,X^{-1}]$.  Mathlib provides the
result \li{IsLocalization.isPrime\_of\_isPrime\_disjoint}, which states that
the extension of a prime ideal to a localization remains prime whenever the
ideal is disjoint from the multiplicative set.  The condition $r\nmid X$
ensures precisely that $(r)$ is disjoint from the powers of~$X$: if
$X^n\in(r)$ for some~$n$, then $r\mid X^n$, so $r\mid X$ (since $r$ is
prime), contradicting the hypothesis.

This lemma illustrates a general pattern in localization-based arguments: the
``interesting'' primes in a localization are those that come from primes in the
base ring whose principal ideals avoid the multiplicative set.  The powers
of~$X$ that are absorbed as units in $R[X,X^{-1}]$ are precisely the
``uninteresting'' primes; all other primes survive the localization.

\subsection{Concrete instantiations}\label{sec:examples}

The file \lf{Applications/Examples.lean} records several concrete consequences
of the polynomial-ring theorem:

\begin{enumerate}[nosep]
\item $\ZZ$ is a UFD (a direct wrapper around Mathlib's instance).
\item $\ZZ[X]$ is a UFD (combining the Mathlib integer UFD instance with the
  polynomial-ring theorem).
\item For any field~$k$, the multivariate polynomial ring
  $k[X_1,\ldots,X_n]$ is a UFD (using Mathlib's
  \li{MvPolynomial.uniqueFactorizationMonoid}).
\end{enumerate}

\noindent
These examples are deliberately thin---each is a one- or two-line invocation
of existing Mathlib infrastructure---but they serve as machine-checkable
witnesses that the formalization connects cleanly to the broader library
ecosystem.  In particular, the $\ZZ[X]$ example can be proved either via the
Gauss lemma (Mathlib's default route) or via the Nagata theorem (our route
through $\ZZ[X,X^{-1}]$), providing a concrete comparison point.  By contrast,
Corollary~\ref{cor:iterpoly} above is a genuinely theorem-driven second
application of the Nagata package itself.

\section{Lessons from the formalization}\label{sec:lessons}

Several aspects of the formal development may be of interest to the
proof-engineering community beyond the specific theorem.

\subsection{Getting the hypothesis right}\label{sec:hyp-right}

The most consequential design decision was the replacement of the
prime-or-unit hypothesis with the prime-generated one.  The prime-or-unit
condition is superficially cleaner---it avoids quantifying over
multisets---but it is essentially degenerate for submonoids with more than one
non-unit prime generator, as discussed in Section~\ref{sec:degeneracy}.

The formalization made this defect visible earlier and more sharply than a
textbook exposition would.  The initial development compiled and type-checked
under the prime-or-unit hypothesis; all transfer lemmas were proved correctly.
The problem surfaced only at the application stage: the powers-of-$X$ submonoid
satisfies the prime-generated condition but \emph{not} the prime-or-unit
condition, since $X^2\in S$ is neither prime nor a unit in~$R[X]$.  In a
pen-and-paper development, one might paper over this gap by implicitly using the
prime-generated formulation in the application while stating the theorem with
the prime-or-unit hypothesis.  The formal setting does not allow this.

The mathematical lesson is that a condition on individual elements of a
submonoid is the wrong abstraction level when the argument requires factoring
products of denominators; one must instead ask for a \emph{factorization
property} of the submonoid.  This insight is not original---textbooks that
treat the general case state the theorem with a prime-factorization hypothesis
from the outset---but the formalization provided a concrete, machine-checkable
failure mode.

\subsection{Transfer lemmas as conditional infrastructure}\label{sec:transfer}

The divisibility, irreducibility, and primality transfer lemmas are not
unconditional equivalences between~$R$ and~$S^{-1}R$.  They hold under the
specific combination of prime-generation and prime-avoidance hypotheses that
arises in the Nagata argument.  Formalizing them as self-contained lemmas
(rather than inlining them into the main proof) required identifying exactly
which hypotheses each transfer result needs.

For example, the divisibility lift (Lemma~\ref{lem:lift-dvd}) requires both
prime-generation of~$S$ and avoidance by~$p$, while the primality transfer
(Lemma~\ref{lem:transfer-prime}) additionally requires that $p/1$ be prime in
the localization.  In the key lemma, the avoidance hypothesis is supplied by
one branch of the case split, and the localization-side primality comes from
the UFD hypothesis---but the transfer lemma itself does not know about either
source.  This decomposition is valuable not because each lemma is deep, but
because the precise statement makes its reuse conditions explicit.

\subsection{Multiset induction as the core combinatorial engine}

The most labor-intensive part of the formalization is the multiset induction in
the transfer lemmas.  When clearing a denominator $s=q_1\cdots q_n$ from a
fraction, the argument peels off one prime factor at a time, using primality
of~$q_i$ to redirect the factor to the appropriate side of a product equation.
This combinatorial bookkeeping---tracking which factors have been consumed,
maintaining the inductive invariant, and handling the base case---is routine on
paper but requires explicit management in the formal setting.

The auxiliary lemma \li{split\_prime\_factors\_of\_mul\_eq} (which partitions a
multiset of prime factors across the two sides of a product equation) is the
longest single proof in the development.  Its proof is a nested induction:
the outer induction is on the multiset, and at each step the primality of the
current factor determines which side of the partition absorbs it.  The
existence of this lemma as a self-contained, reusable result is a
formalization artifact---on paper one would simply say ``by induction on the
number of prime factors''---but it clarifies the precise combinatorial content
of the transfer-of-irreducibility argument.

\subsection{Automation and transparency}

Lean's tactic automation handles routine algebraic reasoning, but the key proof
obligations---particularly the case analysis on irreducible elements dividing
products of primes---require explicit combinatorial arguments beyond
\li{simp} or \li{ring}.  This is a feature, not a limitation: the formal proof
preserves exactly the case distinctions that constitute the mathematical
content.

The most effective automation patterns in this development were:
\begin{itemize}[nosep]
\item \li{simp} with commutativity and associativity lemmas for rearranging
  products.
\item \li{ac\_rfl} for closing goals that differ only by
  associativity/commutativity.
\item Lean's equation compiler for structural recursion on multisets (via
  \li{Multiset.induction\_on}).
\end{itemize}

\subsection{Localization ergonomics}

Working with Mathlib's localization API involves navigating between the abstract
\li{IsLocalization} typeclass and concrete \li{Localization} instances.  The
project uses a two-layer approach: abstract interface lemmas handle
divisibility, units, and surjectivity uniformly across localization targets,
while the concrete quotient remains useful when one wants a canonical target
type with definitional fraction constructors.

A key design choice was the \li{dvd\_of\_iff} lemma, which characterizes
divisibility of images in the localization:
\[
  \iota(a)\mid\iota(b)
  \;\;\Longleftrightarrow\;\;
  \exists\,s\in S,\;a\mid s\cdot b.
\]
This equivalence is the interface through which the Nagata transfer lemmas
access the localization: they never manipulate fractions directly, but instead
use it to translate localization-side divisibility into a ring-side equation
amenable to multiset induction.

The resulting architecture is genuinely mixed rather than wrapper-only.
\lf{Localization/IsLocalization.lean} provides the abstract API.
\lf{Nagata/Lemmas.lean} carries the transport chain at the
\li{IsLocalization} level, and the concrete \li{Localization\,S} theorems are
recovered as specializations.  This keeps the downstream application files
short while making the core theorem family available from arbitrary
localizations.

\subsection{Formalization effort and Lean~4 ergonomics}

The current formalization comprises 1546 lines of Lean~4 source across 18
files, with 97 formally verified theorem/lemma declarations.  It contains zero
uses of \li{sorry}, \li{admit}, \li{axiom}, or \li{native\_decide}.  The
development targets Lean~4.24.0 and Mathlib as pinned by
\lf{lean-toolchain} and \lf{lake-manifest.json}.

Several Lean~4-specific ergonomic issues shaped the proof style:

\begin{itemize}[nosep]
\item \textbf{Typeclass diamonds.}  The \li{Localization} type carries
  multiple paths to the commutative-monoid-with-zero structure, which
  occasionally forces the use of explicit \li{@} notation to disambiguate
  instances.  The theorem statements in Section~\ref{sec:lean-statements} show
  this pattern.
\item \textbf{Multiset induction.}  Key transfer lemmas reason by induction on
  the multiset of prime factors.  Lean's \li{Multiset.induction\_on} tactic
  handles the base and cons cases, but the recursive step often requires manual
  \li{calc} blocks to guide the simplifier.
\item \textbf{Universe polymorphism.}  All definitions and theorems are
  universe-polymorphic, which is necessary for eventual Mathlib compatibility
  but requires care in stating auxiliary lemmas to avoid universe
  metavariable errors.
\item \textbf{Zero exclusion.}  Several transfer lemmas require $0\notin S$ as
  a hypothesis.  In Lean~4, this is expressed as a \li{Fact} instance, which
  must be registered in the local context before typeclass inference can
  construct the localization's \li{IsDomain} instance.  This boilerplate
  appears in the key lemma and in the main theorem's proof.
\end{itemize}

\section{Related work}\label{sec:related}

\subsection{Mathematical background}

The result first appeared in a short note by Nagata~\cite{nagata1957}, which
proved that a noetherian domain whose localization at a multiplicatively closed
set of primes is a UFD is itself a UFD.  Nagata subsequently included the
theorem in his monograph on local rings~\cite{nagata1962}.  Samuel's Tata
lectures~\cite{samuel1964} give a particularly clean treatment of the
single-prime-generator case and the polynomial-ring application.  Matsumura's
textbook~\cite{matsumura1989} and Kaplansky's \emph{Commutative
Rings}~\cite{kaplansky1970} present the result in full generality for
submonoids generated by primes, and our formulation follows this tradition.

The Stacks Project~\cite{stacks} provides a modern, self-contained exposition
of the theorem (Tag~0AFU) in the language of commutative algebra, including a
proof that any noetherian normal domain whose class group is trivial is a UFD.
We followed the argument structure closest to Matsumura's and Samuel's accounts.

\subsection{The Lean theorem prover and Mathlib}

Our development is built in Lean~4~\cite{moura-lean4}, a dependently typed
functional programming language and interactive theorem prover.  The
mathematical foundation is provided by Mathlib~\cite{mathlib}, the
community-maintained library of formalized mathematics for Lean~4.

Mathlib supplies all the algebraic infrastructure on which this project rests:
the \li{Localization} type and \li{IsLocalization} typeclass; the
\li{UniqueFactorizationMonoid} class and its characterizations; the
\li{WfDvdMonoid} instance for noetherian domains; and the polynomial and
Laurent polynomial rings with their ring-theoretic structure.  Our contribution
is the formal packaging of a Nagata-style transfer theorem and the proof
architecture needed to make it applicable.  The project depends on Mathlib but
does not modify it.

\subsection{Formal algebra in other proof assistants}

In Coq, the Mathematical Components library~\cite{mathcomp} offers a mature
algebraic environment with strong support for divisibility and polynomial
algebra, including infrastructure for Gauss-lemma-style and UFD-style arguments
over polynomial rings.  The algebraic hierarchy is built on the SSReflect proof
style, which favors boolean decision procedures over classical logic.

In Isabelle/HOL, Bordg~\cite{bordg-localization} formalizes localization of
commutative rings in the Archive of Formal Proofs, providing the
ring-of-fractions construction and its universal property.
Divas{\'o}n et al.~\cite{divason-berlekamp} formalize the
Berlekamp--Zassenhaus factorization algorithm, demonstrating substantial
factorization infrastructure, though their focus is algorithmic factorization
rather than structural factoriality theorems.

\subsection{Related formalization projects in Lean}

The closest thematic relative in Lean/Mathlib is the formalization of Dedekind
domains and class groups by Baanen et al.~\cite{baanen-dedekind}, which
also develops nontrivial algebraic infrastructure on top of Mathlib's
localization and ideal theory.  Their work includes formalized proofs about
fractional ideals, the class group of a global field, and the finiteness of
the class number, all of which require careful interaction with localization
machinery similar to that needed for the Nagata theorem.

Among recent Annals of Formalized Mathematics publications, Brasca
et al.~\cite{brasca-flt} formalize Fermat's Last Theorem for regular primes
in Lean~4, including a proof of Kummer's lemma via Hilbert's Theorems~90--94.
Riou~\cite{riou-derived} formalizes derived categories and triangulated
structures in Lean/Mathlib.  Loeffler and Stoll~\cite{loeffler-zeta}
formalize the theory of zeta and $L$-functions, including a proof of
Dirichlet's theorem on primes in arithmetic progressions.  These projects
share the methodological characteristic of building substantial mathematical
results on top of Mathlib's existing infrastructure, and they illustrate the
current state of the art for formal commutative algebra and number theory in
Lean~4.

\subsection{Novelty statement}

No public formalization of Nagata's factoriality theorem is known to us in
Lean/Mathlib, Coq/MathComp, or Isabelle/AFP.  This is a statement about the
public state of these libraries at the time of writing, not a claim that no
unpublished or in-progress formalization exists.

To sharpen the distinction: Bordg's Isabelle work and Mathlib's own
localization API provide generic \emph{localization infrastructure} but no
factoriality transfer theorem.  The MathComp Gauss-lemma proof establishes
$R[X]$ is a UFD by a different route that does not involve localization at all.
Baanen et al.'s Dedekind-domain formalization shares localization machinery but
targets class groups and fractional ideals, not factoriality descent.  The
present work fills the specific gap of a reusable \emph{localization-to-base}
UFD transfer theorem.

\section{Prospective library contributions}\label{sec:upstream}

We classify the results by Mathlib readiness.  The localization helper API is
directly upstreamable, and the \li{PrimeGenerated} predicate plus its closure lemmas are natural first
candidates.  Since the transport chain and theorem family are now formulated at
the abstract \li{IsLocalization} level, the remaining Mathlib question is
primarily one of packaging rather than refactoring.

A concrete split would be:
\begin{enumerate}[nosep]
\item \textbf{PR1: prime-generated infrastructure.}  Move
  \li{PrimeGenerated}, \li{Avoids}, the zero-exclusion and powers/closure
  lemmas, and the multiset bookkeeping that does not depend on localization.
  In the current artifact, this is the generic portion of
  \lf{Nagata/Lemmas.lean} together with the supporting lemmas from
  \lf{Localization/MultSet.lean}.
\item \textbf{PR2: abstract localization transport API.}  Move the helper
  localization lemmas together with the abstract divisibility,
  irreducibility, and primality transport lemmas.  Concretely, this includes
  the \li{surj}/\li{mk'\_eq\_iff}/\li{dvd\_map\_iff} helper layer and the
  \li{\_isLocalization}-suffixed transfer lemmas for the prime-generated and
  prime-or-unit chains; in the current artifact these live in
  \lf{Localization/IsLocalization.lean} and \lf{Nagata/Lemmas.lean}.
\item \textbf{PR3: Nagata theorem family.}  Move the abstract theorem family in
  \lf{Nagata/Theorem.lean}, headed by
  \li{nagata\_theorem\_isLocalization}, together with the closure and
  finite-prime-generator corollaries.  The application files can remain outside
  the PR sequence as downstream demonstrations.
\end{enumerate}

\noindent
We do not claim that these PRs have already been opened.  The point is that the
artifact now supports a reviewable split rather than a single monolithic
upstreaming story.  The Nagata theorem itself remains the paper-level
contribution; whether it belongs in Mathlib depends on community demand.

\section{Artifact and reproducibility}\label{sec:artifact}

\subsection{Artifact summary}

The complete formalization is publicly available at:
\begin{center}
\url{https://github.com/Arthur742Ramos/NagataFactoriality}
\end{center}
The artifact described here is the repository snapshot accompanying this
manuscript.  The development targets Lean~4.24.0, builds with Lake (the Lean
build system), and contains no \li{sorry}, \li{admit}, or \li{axiom}
placeholders.
Table~\ref{tab:stats} summarizes the artifact statistics.

\begin{table}[ht]
\centering
\caption{Artifact statistics.}\label{tab:stats}
\smallskip
\begin{tabular}{@{}lr@{}}
\toprule
\textbf{Metric} & \textbf{Value} \\
\midrule
Lean source files              & 18   \\
Lines of Lean source           & 1546 \\
Theorem/lemma declarations     & 97   \\
\li{sorry}/\li{admit}/\li{axiom} & 0   \\
Lean version                   & 4.24.0 \\
\bottomrule
\end{tabular}
\end{table}

\subsection{Building the artifact}

The artifact builds in two steps:
\begin{lstlisting}[style=lean]
lake exe cache get     -- fetch prebuilt Mathlib oleans
lake build             -- build the project
\end{lstlisting}
The first command fetches prebuilt object files for the Mathlib dependency,
avoiding a full rebuild of the library.  The second command compiles all
project files, including the application layer.

\subsection{Continuous integration}

The repository includes a GitHub Actions workflow that runs \li{lake build} on
every push and pull request, ensuring that the development remains buildable
against its pinned Lean toolchain and Mathlib version.

\subsection{Release packaging}

The submission artifact is a tagged source release that freezes the toolchain
pin, lock file, Lean sources, and this paper.  Citation metadata is provided in
a \li{CITATION.cff} file following the Citation File Format standard.  The same
release can be mirrored to a DOI-minting archive such as Zenodo if the venue
accepts archival artifacts.

\subsection{Verification and trust}

The trusted base for the formalization consists of the Lean~4 kernel (which
checks all proof terms), the Mathlib library (which provides algebraic
infrastructure), and the Lake build system (which orchestrates compilation).
No external oracles, custom axioms, or unverified code generation are used.
The absence of \li{sorry} and \li{admit} can be mechanically verified by
searching the source files.

\section{Conclusion and future work}\label{sec:conclusion}

We have presented, to our knowledge, the first public Lean~4 formalization of
Nagata's factoriality theorem under a prime-generated hypothesis on the
multiplicative set.  The contribution is organized around a reusable theorem
surface that combines concrete \li{Localization\,S} statements, abstract
\li{IsLocalization} theorem families, and packaged prime-generator entry
points.  The same interface supports two Nagata-based proofs of the polynomial
UFD theorem---via Laurent localization and via the fraction field---as well as
an iterated polynomial corollary obtained by direct reuse.  The formalization
also exposed that the prime-or-unit variant is too weak as a general theorem
statement; the corrected prime-generated formulation is the one that supports
reusable applications.

\paragraph{Limitations.}
The formalization covers the multiplicative version of Nagata's theorem (for
submonoids generated by primes).  The more general ideal-theoretic
formulations---such as the version involving $t$-operations or Krull
domains---are not addressed.  The two polynomial applications are intended to
demonstrate reuse of the theorem package rather than to exhaust its range of
applications; broader downstream case studies would test the abstraction
further.

\paragraph{Future work.}
Several directions are natural continuations:
\begin{enumerate}[nosep]
\item \textbf{Mathlib integration.}  Pursue the concrete three-stage split from
  Section~\ref{sec:upstream}: prime-generated infrastructure first, then the
  abstract localization transport API, and finally the Nagata theorem family.
  \item \textbf{Richer applications.}  Beyond the polynomial and iterated-polynomial
  corollaries formalized here, the theorem could be applied to establish unique
  factorization in power-series rings, rings of algebraic integers, or other
  settings where a Nagata-style argument is natural.
\item \textbf{Generalization.}  The ideal-theoretic generalization of Nagata's
  theorem (involving the class group and divisorial ideals) is a natural next
  target, though it would require substantially more infrastructure than is
  currently available in Mathlib.
\item \textbf{Comparison with Gauss-lemma approaches.}  A formal comparison of
  the Nagata-based and Gauss-lemma-based proofs of the polynomial UFD theorem,
  both in proof complexity and dependency structure, would be an interesting
  study in formalization methodology.
\end{enumerate}


\end{document}